\documentclass[11pt,
]{amsart}

\usepackage{
amssymb, graphicx
}

\date{\today}

\newtheorem{theorem}{Theorem}

\newtheorem{lemma}{Lemma}
\newtheorem{definition}{Definition}

\theoremstyle{definition}
\newtheorem{remark}{Remark}[section]

\DeclareMathOperator{\WF}{WF}

\newcommand{\eps}{\varepsilon}

\newcommand{\R}{{\bf R}}
\newcommand{\Id}{\mbox{Id}}
\renewcommand{\r}[1]{(\ref{#1})}
\newcommand{\PDO}{$\Psi$DO}
\newcommand{\be}[1]{\begin{equation}\label{#1}}
\newcommand{\ee}{\end{equation}}

\renewcommand{\d}{\mathrm{d}}

\renewcommand{\i}{\mathrm{i}}

\newcommand{\bo}{\partial \Omega}

\title[Local recovery of the compressional and shear speeds]{Local recovery of the compressional and shear speeds from the hyperbolic DN map}

\author[P. Stefanov]{Plamen Stefanov}
\address{Department of Mathematics, Purdue University, West Lafayette, IN 47907}

\author[G. Uhlmann]{Gunther Uhlmann}
\address{Department of Mathematics, University of Washington, Seattle, WA 98195, Department of Mathematics University of Helsinki, Finland, IAS, HKUST, Clear Water Bay, Hong Kong, China}

\author[A. Vasy]{Andras Vasy}
\address{Department of Mathematics, Stanford University, Stanford  CA 94305}

\thanks{The authors were partially supported by the National Science Foundation under
grant DMS-1600327 (P.S.), DMS-1265958 (G.U.) and
DMS-1361432 (A.V.).}

\begin{document}
\begin{abstract}
We study the isotropic elastic wave equation in a bounded domain with boundary. We show that  local knowledge of the Dirichlet-to-Neumann map determines uniquely the speed of the p-wave locally if there is a strictly convex foliation with respect to it, and similarly for the  s-wave speed.  
\end{abstract} 
\maketitle

\section{Introduction}  
Consider the isotropic elastic wave equation in a smooth bounded domain $\Omega$. We study the following problem: can we determine the Lam\'e parameters $\lambda$, $\mu$ and the density $\rho$ from the knowledge of the Dirichlet-to-Neumann $\Lambda$ (DN) map on the boundary? In fact, we are interested in the local problem: local recovery of those parameters from local or even microlocal information about $\Lambda$. Our main motivation is the local seismology problem of recovery of the inner structure of the Earth from local measurements on its surface  of seismic waves. 

This problem is well studied for the wave equation $(\partial_t^2-\Delta_g)u=0$ related to a Riemannian metric $g$ with either full or partial boundary data.  It is known that one can recover $g$ up to an isometry fixing $\bo$ pointwise  \cite{BelishevK92}, using the boundary control method developed by \cite{Belishev_87}. The latter relies on the  unique continuation result of Tataru \cite{tataru95}. We refer to  \cite{KathcalovKL-book} for related results and more references. Logarithmic type of stability is proved in \cite{BKL, LaurentL2017}.
       H\"older type of stability estimates with full data have been proven in  \cite{SU-IMRN, Carlos_12, BellassouedDSF,BaoZhang} and most recently in \cite{SUV-DNmap2014}, under some assumptions on the metric, for example absence of conjugate points. 

In the elastic case, the results are less complete. Unique continuation holds \cite{Nakamura_UC} but the boundary control method is not known to work, see, e.g., \cite{belishev_2007}. The reason is that it is not possible, or at least not known how to decouple the elasticity system completely even though it is easy to do that on the principal symbol level or even for the full symbol, see \r{VEU} below, but only microlocally.  A Lam\'e type of system having the same principal part which can be decoupled fully was studied in \cite{Belishev2006} and the boundary control method was used for it for a unique recovery of the two wave speeds locally with a local data. 
 Numerical reconstruction is proposed in	\cite{Lechleiter-S-2017}.
 
 Rachelle \cite{Rachele_2000} proved that one can recover  the jet of $\lambda$, $\mu$ and $\rho$ at $\bo$ explicitly. In \cite{Rachele00,Rachele03}, she showed that one can recover those three parameters in $\Omega$ provided that $\Lambda$ is known on the whole boundary and assuming that the two wave speeds are simple (strict convexity and no conjugate points). The proof is based on recovering the lens relations related to the two speeds and then applying known rigidity results.  The recovery of  all the three parameters in \cite{Rachele03} requires the study of the second order term and an inversion of the geodesic ray transform. The second author and Hansen \cite{HansenUhlmann03} studied this problem with a residual stress and without the assumption of no-conjugate points or caustics and showed that one can recover both lens relations and derive several consequences of that.

In this work we show that one can recover uniquely the two wave speeds $c_p$ and $c_s$ locally under the assumption of existence of a strictly convex foliation. This condition allows for conjugate points, see, e.g., \cite{Monard14}.  
If $\Omega$ is a ball and the speeds increase when the distance to the center decreases (typical for geophysical applications), the foliation condition is satisfied, see section~\ref{sec_ex}. To prove the main result, we show that one can recover the lens relations related to the two speeds in an explicit way and then apply the results of the authors \cite{SUV_localrigidity}, see also \cite{SUV_anisotropic}, about local recovery of a sound speed in the acoustic equation from the associated lens relation, also known locally. That argument also implies stability as a consequence of the stability result in \cite{SUV_localrigidity}  
but we do not make this formal. Also, we can apply the result if there is an internal closed strictly convex surface where the coefficients jump, like in the elastic Earth model, and recover the two speeds between the boundary and that surface. Indeed,  for that we only need the lens relation along geodesics not hitting that surface; and that can be extracted from the microlocal support of the DN map. 
Note that this recovers two quantities depending on the three parameters $\lambda$, $\mu$ and $\rho$. Recovery of all three parameters would require an analysis of the next order term in the geometric optics construction, similarly to what is done in \cite{Rachele03}.

\section{Main Result}

The isotropic elastic system in a smooth bounded domain $\Omega\subset \R^3$ is described as follows. The elasticity tensor is defined by
\[
c_{ijkl} = \lambda \delta_{ij}  \delta_{kl} +\mu(\delta_{ik}\delta_{jl} + \delta_{il}\delta_{jk}),
\]
where $\lambda>0$, $\mu>0$ are the Lam\'e parameters. The elastic wave operator is given by
\[
(Eu)_i = \rho^{-1}\sum_{jkl} \partial_j c_{ijkl} \partial_l u_k,
\]
where $\rho>0$ is the density and the vector function $u$ is the displacement.

The operator $E$ is symmetric on $L^2(\Omega;\mathbf{C}^3,\rho\,\d x)$. It has has a principal symbol
\be{s0}
\sigma_p(-E)v  = \frac{\lambda+\mu}{\rho} \xi \xi\cdot v + \frac{\mu}{\rho} |\xi|^2 v,\quad v\in\mathbf{C}^n.
\ee
Taking $v=\xi$ and $v\perp\xi$, we recover the well known fact that that $\sigma_p(-E)$ has eigenvalues 
\[
c_p= \sqrt{(\lambda+2\mu)/\rho}, \quad c_s = \sqrt{\mu/\rho}
\]
of multiplicities $1$ and $2$, respectively and eigenspaces $\R\xi$, and $\xi^\perp$. Those are known as the speeds of the p-waves and the s-waves, respectively.  
 The eigenspaces correspond to  the polarization of those waves. The characteristic variety $\det \sigma_p(E) =0$ is the union of $\Sigma_p := \{\tau^2=c_p^2|\xi|^2\}$ and $\Sigma_s := \{\tau^2=c_s^2|\xi|^2\}$, each one  having two connected components (away from the zero section), determined by the sign of $\tau$. 

Let $u$ solve the elastic wave equation 
\be{1}
\begin{cases}
u_{tt} -Eu &=0\quad \text{in $\R\times\Omega$},\\ 
 \ u|_{\R\times\bo}   &=f,\\
\ \ \ \  u|_{t<0}&=0,
\end{cases}
\ee
with $f$ given so that $f=0$ for $t<0$. The Dirichlet-to-Neumann $\Lambda$ map is defined by
\be{2a}
(\Lambda f)_i = Nu:=  \Sigma_j \sigma_{ij}(u)\nu^j\big|_{\bo},
\ee
where $\nu$ is the outer unit normal on $\bo$, and $\sigma_{ij}(u)$ is the stress tensor
\be{1s}
\sigma_{ij}(u) = \lambda \nabla\cdot u\delta_{ij} + \mu(\partial_j u_i + \partial_i u_j). 
\ee
Note that $Eu = \rho^{-1}\delta\sigma(u)$, where $\delta$ is the divergence of the 2-tensor $\sigma(u)$.

Let $\Omega_{\rm ext}$ be an open domain containing $\bar\Omega$ and extend the coefficients there in a smooth way. 
\begin{definition}\label{def_5.1}
  Let $\kappa: \Omega_{\rm ext}\to[0,\infty)$ be  a smooth function which level sets $\kappa^{-1}(q)$, $q\le 1$, restricted to $\bar\Omega$, are strictly convex viewed from $\kappa^{-1}((0,q))$ w.r.t.\ $g$;   $d\kappa\not=0$ on these level sets,   $\kappa^{-1}(0)\cap \bar\Omega\subset\bo$, and $M_0\subset \kappa^{-1}([0,1])$. Then we call $\kappa^{-1}([0,1])\cap \bar\Omega$ a strictly convex foliation of $\bar\Omega$ w.r.t.\ $g$. 

\end{definition}

A special case is when $\bo$ is strictly convex w.r.t.\ $g$ and  $\kappa^{-1}(0)=\bo$.

We introduce the lens rigidity problem next. For a compact manifold $(M,g)$ with a boundary,  let the manifolds $\partial_\pm SM$ consist of all vectors $(x,v)$ with $x\in\partial M$, $v$ unit in  the metric $g$,  and  pointing outside/inside $M$. 
We define the \textit{scattering relation}
\begin{equation} \label{L}
L: \partial_- SM \longrightarrow \partial_+SM 
\end{equation}
in the following way: for each $(x,v)\in \partial_- SM$, $L(x,v)=(y,w)$, where $(y,w)$ are the exit point and  direction, if exist, of the maximal unit speed geodesic $\gamma_{x,v}$ in the metric $g$, issued from $(x,v)$. Let 
\[
\ell: \partial_-SM \longrightarrow \R\cup\infty
\]
be its length, possibly infinite. If $\ell<\infty$, we call $M$ non-trapping. The maps $(L,\ell)$ together are called \textit{lens relation} (or lens data).  

It is convenient to identify vectors in $\partial_- SM$ with their projections on the unit ball bundle $B(\partial M)$; and similarly for $\partial_+SM$. Then we can view $L$ and $\ell$ as maps from $B(\partial M)$ to itself; or from $B(\partial M)$ to $\R\cap\infty$, respectively. 

Below, we denote by $L_p$ and $L_s$ the lens relations in $M=\bar\Omega$ w.r.t.\ the metrics $c_p^{-2}\d x^2$ and $c_s^{-2}\d x^2$, respectively. Similarly, we denote the corresponding $\ell$'s by $\ell_p$ and $\ell_s$.

\begin{theorem}\label{thm1}
Let $\rho$, $\lambda$, $\mu$ be smooth in $\bar\Omega$. 
Let $\kappa^{-1}(q)$, $q\in [0,1]$ be a strictly convex foliation w.r.t.\ $c_p$, and let  $\Gamma=\kappa^{-1}([0,1])\cap  \bo$.  Then for every $\eps>0$,  $c_p$ is uniquely determined on the foliation $\kappa^{-1}([0,1])\cap  \bar\Omega$    by knowledge, up to a smooth function, of the kernel $\Lambda(t_2,x_2,t_1,x_1)$ of  $\Lambda$ on $(0,T)\times\Gamma\times(0,\eps) \times\Gamma$, if $T$ is greater than the length of all geodesics, in the metric $c_p^{-2} \d x^2$,  in $\bar\Omega$  having the property that each one is tangent to some of the hypersurfaces 
in the foliation. 

The same statement remains true for $c_p$ replaced by $c_s$. 
\end{theorem}

We refer to Figure~\ref{pic} for an illustration of the set where we can recover the speeds. 

\begin{figure}
\includegraphics[page=1,scale=0.7]{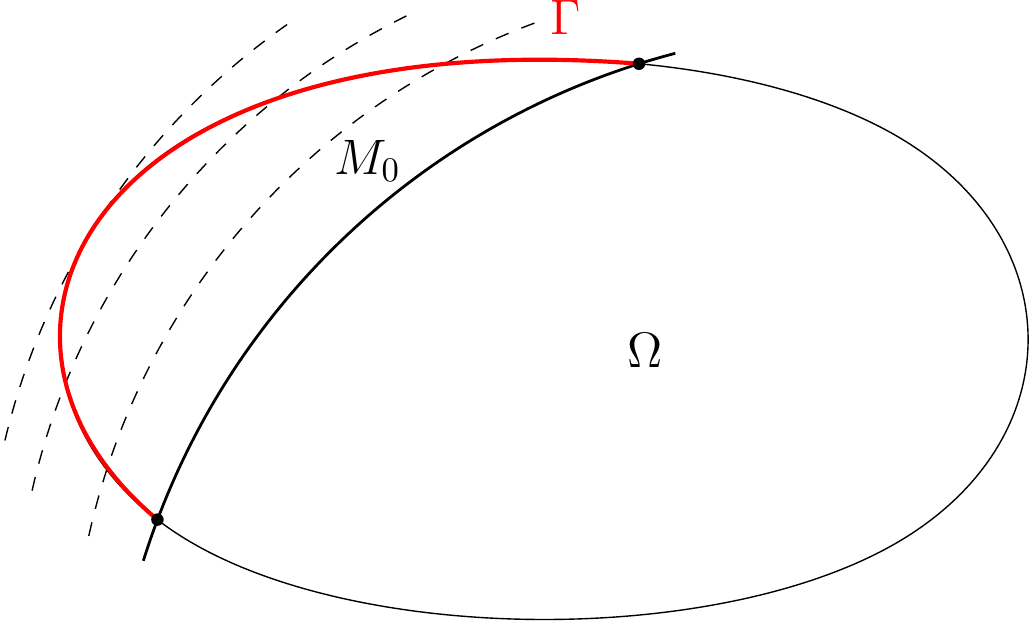}
\caption{The foliation related to either $c_p$ and $c_s$, and the part $\Gamma$ of $\bo$ where $\Lambda$ is known. We can recover the speed in $M_0$.}\label{pic}
\end{figure}

We only need microlocal information about $\Lambda$, wee Remark~\ref{remark5.1}. Also, 
since $(L_p,\ell_p)$ and $(L_s,\ell_s)$ are part of the wave front set of the kernel of $\Lambda$, we can argue that they can be recovered stably form it; and in fact, they are directly observable in seismic experiments. Then one can apply the stability result in \cite{SUV_localrigidity} to conclude that $c_s$ and $c_p$ are stably recoverable from $\Lambda$. We will not make this statement more precise in this paper.

\section{Preliminaries}
\subsection{An invariant formulation} 
We have 
\be{E}
(Eu)_i = \rho^{-1}\sum_{j} (  \partial_i\lambda\partial_j u_j + \partial_j\mu \partial_j u_i + \partial_j\mu \partial_i u_j     ). 
\ee 
This can also be written in the following divergence form. 
\be{L0}
Eu = \rho^{-1}(  \d \lambda \delta u+ 2 \delta \mu \d^s u   ),
\ee
where $\d^su=(\partial_ju_i+ \partial_iu_j)/2$ is the symmetric differential, and $\delta= -(\d^s)^*$ is the divergence of symmetric fields. 

To prepare ourselves for changes of variables needed in the analysis near surfaces that we will flatten out, we will write $E$ in invariant way in the presence of a Riemannian metric $g$. We view $u$ as an one form (a covector field) and we define the symmetric differential $\d^s$ and the divergence $\delta$ by
\[
(\d^s u)_{ij}= \frac12\left(\nabla_i u_j+\nabla_j u_i\right), \quad (\delta v)_i = \nabla^j v_{ij},\quad \delta u = \nabla^iu_i,
\]
where $\nabla$ is the covariant differential, $\nabla^j = g^{ij}\nabla_i$, $u$ is a covector field, and  $v$ is a symmetric covariant tensor field of order two.   Note that $\d^s$ increases the order of the tensor by one while $\delta$ decreases it by one. Then we define $E$ by \r{L0}. The stress tensor \r{1s} is given by
\be{1s2}
\sigma(u) = \lambda (\delta u)g + 2\mu \d^s u,
\ee
and then $Eu=\rho^{-1}\delta\sigma(u)$. 
The Neumann boundary condition $Nu$ at $\bo$ is still given by prescribing the values of $\sigma_{ij}(u)\nu^j$ on it as in \r{2a}.   
The operator $E$, defined originally on 
$C_0^\infty(\Omega)$  extends to a self-adjoint operator in $L^2(\Omega, \rho\d x)$. This extension is the one satisfying the zero Dirichlet boundary condition on $\R\times\bo$. In particular, this shows that the mixed problem \r{1} is solvable with smooth data $f$ at least since one can always extend $f$ inside and reduce the problem to solving one with a zero boundary condition and a non-zero source term; and then use the Duhamel's principle for the latter. 

We show next that the data $(u,Nu)$ on the boundary is equivalent to knowing the Cauchy data on it, see also \cite[sec.~3.1]{Rachele_2000}. In next lemma, we use semigeodesic coordinates $x=(x',x^3)$ to a given hypersurface $S$, with $x^3>0$ on one side of it, defining the orientation. 
The Euclidean metric then takes the form $g$ in those coordinates with $g_{\alpha 3}=\delta_{\alpha 3}$ for $1\le\alpha\le 3$. 

\begin{lemma}\label{lemma_N}
For every hypersurface $S$, the pair 
$(u,Nu)|_{S}$ determined uniquely the Cauchy data $(u,\partial_\nu u)|_S$. More precisely, in semigeodesic coordinates,  the normal derivative of $u$ on $S$ can be obtained from $u|_S$ and $Nu|_S$ by the relations
\be{NN}
\begin{split}
\partial_3 u_\alpha  &=  \frac1\mu (Nu)_\alpha  -\partial_\alpha u_3 + 2 \Gamma_{j3}^ku_\alpha   , \quad \alpha=1,2,\\
\partial_3u_3 &= \frac1{2\mu} \left( (Nu)_3 - \lambda (\delta u)  \right),
\end{split}
\ee
where $\Gamma_{ij}^k$ are the Christoffel symbols of the Euclidean metric in semigeodesic coordinates. 
\end{lemma}

Before presenting the proof, we want to emphasize that $u$ is also transformed in the new coordinates (as a covector), and $\delta$ is the divergence w.r.t.\ the transformed metric. Also, in a more invariant form, \r{NN} takes the form
\be{NNa}
\begin{split}
\nabla_3 u_\alpha  &=  \frac1\mu (Nu)_\alpha  -\nabla_\alpha u_3   , \quad \alpha=1,2,\\
\partial_3u_3 &= \frac1{2\mu} \left( (Nu)_3 - \lambda (\delta u)  \right),
\end{split}
\ee
where $\nabla$ is the covariant derivative, and we used the fact that $\nabla_3u_3=\partial_3u_3$. 

\begin{proof}
In those coordinates,
\[
(Nu)_j = \lambda (\delta u) \delta_{j3} + \mu\left( \partial_3 u_j + \partial_j u_3- 2 \Gamma_{j3}^ku_k\right).
\]
 Setting $j=3$, we get the second  formula in \r{NN} because $\Gamma_{33}^k=0$. Taking  $j=1,2$, we get the first one one. 
\end{proof}

\section{Geometric optics for the elastic wave equation} \label{sec_GO} 
We recall the well known geometric optics construction for the acoustic and the elastic wave equations, see, e.g., \cite{Taylor-book0, Treves}.

\subsection{The Cauchy Problem with data at $t=0$ in the acoustic case} We start with the scalar acoustic case which we use in the analysis of  the elastic one. We work in arbitrary dimensions $n\ge2$ here. Details can be found in \cite{ Taylor-book0, Treves}, for example. 

We recall briefly the geometric optic construction for the acoustic wave equation 
\be{ac}
(\partial_t^2- c^2 \Delta_{g_0})u=0
\ee
 with Cauchy data $(u,\partial_t u)=(h_1,h_2)$ at $t=0$. Here, $g_0$ is a Riemannian metric that we include in order to have the flexibility to change coordinates easily; and $\Delta_{g_0}$ is the Laplace-Beltrami operator. Up to lower order terms, $c^2\Delta_{g_0}$ coincides with $\Delta_g$ with  $g:=c^{-2}g_0$. 
We are looking for solutions of the form 
\be{o1}
\begin{split}
u(t,x) =  (2\pi)^{-n} \sum_{\sigma=\pm}\int e^{\i\phi_\sigma(t,x,\xi)} &\Big( a_{1,\sigma}(t,x,\xi) \hat h_1(\xi)\\
&+  |\xi|_{g_0}^{-1}a_{2,\sigma}(t,x,\xi) \hat h_2(\xi)\Big) \d \xi,
\end{split}
\ee
modulo  terms involving smoothing operators of $h_1$ and $h_2$, defined in some neighborhood of $t=0$, $x=x_0$ with some $x_0$.  This parametrix differs from the actual solution by a smoothing operator applied to $\mathbf{h}=(h_1,h_2)$, as it follows from standard hyperbolic estimates. 

Here, $a_{j,\sigma}$ are classical amplitudes of order zero depending smoothly on $t$ of the form
\be{a}
a_{j,\sigma} \sim \sum_{k=0}^\infty a_{j,\sigma}^{(k)},\quad \sigma=\pm, \; j=1,2,
\ee
where $a_{j,\sigma}^{(k)}$ is homogeneous in $\xi$ of degree $-k$ for large $|\xi|$. 
The phase functions $\phi_\pm$ are positively homogeneous of order $1$ in $\xi$ solving the eikonal equations
\be{o2}
\partial_t\phi\pm c(x)|\nabla_x\phi|_{g_0}=0, \quad 
\phi_\pm|_{t=0}=x\cdot\xi.
\ee
Such solutions exist locally only, in general.

Equate  the order $1$ terms  in the expansion of $(\partial_t^2-c^2\Delta_{g_0})u$  to get that the principal terms of the amplitudes must solve the \emph{transport equation}
\be{trans}
\left( (\partial_t \phi_\pm) \partial_t - c^{2} g_0^{ij}(\partial_i \phi_\pm)\partial_j+C_\pm \right)a_{j,\pm}^{(0)}=0, 
\ee
with appropriate initial conditions and 
\be{tr2}
2C_\pm = (\partial_t^2-c^2\Delta_{g_0})\phi_\pm.
\ee
Equating terms homogeneous in $\xi$ of lower order we get transport equations for $a_{j,\sigma}^{(k)}$, $j=1,2,\dots$ with the same left-hand side as in \r{trans}  with a right-hand side determined by $a_{k,\sigma}^{(k-1)}$.

The transport equations are ODEs along the zero bicharacteristics, which are just the geodesics of the metric $g$ lifted to the phase space, with vectors identified by covectors by the metric. The integrals appearing in \r{o1} are Fourier Integral Operators (FIOs) either with $t$ considered as a parameter, or as $t$ considered as one of the variables. In the former case, singularities of $(h_1,h_2)$ propagate along the zero bicharacteristics. More precisely,  for every $t$, 
\be{C1}
\WF(\mathbf{u}(t,\cdot)) = C_+(t)\circ\WF(\mathbf{h}) \cup C_-(t)\circ\WF(\mathbf{h}),
\ee
where $\mathbf{u}:=(u,u_t)$, $\mathbf{h}=(h_1,h_2)$ and 
\[
\begin{split}
C_+(t)(x,\xi) &= \left( \gamma_{x,\xi/|\xi|_g}(t),|\xi|_g  g\dot \gamma_{x,\xi/|\xi|_g}(t)  \right), \\
C_-(t)(x,\xi)& = \left( \gamma_{x,-\xi/|\xi|_g}(t), -|\xi|_gg\dot \gamma_{x,-\xi/|\xi|_g}(t)  \right) = C_+(-t)(x,\xi),
\end{split}
\]
and for $(x,\eta)\in T^*\R^3\setminus 0$, $\gamma_{x,\eta}$ is the geodesic issued from $x$ in direction $g^{-1}\eta$. 
 
On the other hand, considering $t$ as one of the variables, 
\be{C2}
\WF(\mathbf{u}) = C_+\circ\WF(\mathbf{h}) \cup C_-\circ\WF(\mathbf{h}),
\ee
where 
\[
\begin{split}
C_+(x,\xi) &= \left\{\left( t ,\gamma_{x,\xi/|\xi|_g}(t), -|\xi|_g, |\xi|_gg\dot \gamma_{x,\xi/|\xi|_g}(t)  \right),\; t\in\R\right\}, \\
C_-(x,\xi)& = \left\{\left( t,\gamma_{x,-\xi/|\xi|_g}(t), |\xi|_g, -|\xi|_gg\dot \gamma_{x,-\xi/|\xi|_g}(t)\right)\; t\in\R\right\}.
\end{split}
\]
 In the analysis below, we will consider $C_+$ only. 
 
The construction  above can be done  in some neighborhood of a fixed point $(0,x_0)$ in general. To extend it globally, we can localize it first for $\mathbf{h}$ with $\WF(\mathbf{h})$ in a conic neighborhood of some fixed $(x_0,\xi^0)\in T^*\R^3\setminus 0$. 
Then $u$ will be well defined near the geodesic issued from that point but in some neighborhood of $(0,x_0)$ in general. We can fix some $t=t_1$ at which $u$ is still defined,  take the Cauchy data there and use it to construct a new solution. Then we get an FIO which is a composition of the two local FIOs each one associated with a canonical diffeomorphism, then so is the composition. Then we can use a partition of unity to conclude that while the representation \r{o1} is local, the conclusions \r{C1} and \r{C2} are global. In fact, it is well known that both $\mathbf{h}\mapsto \mathbf{u}$ and  $\mathbf{h}\mapsto \mathbf{u}(t,\cdot)$ with $t$ fixed are global FIOs associated with the canonical relations in \r{C1} and \r{C2}. 
 
 In particular, if $S$ is a smooth hypersurface, and $\gamma_{x,\xi}(t)$ hits $S$ for the first time $t=t(x,\xi)$ transversely locally, then $\mathbf{h}\mapsto u|_S$ is an FIO again with a canonical relation as $C_+$ above but with $t=t(x,\xi)$ and $\dot\gamma$ replaced by its tangential projection $\eta':= \dot\gamma'$. Notice that $\tau=-|\xi|_g<0$ for $C_+$ and $\tau=|\xi|_g>0$ for $C_-$. Also, $|\tau'|<|\eta|_g$ with equality for tangent rays which we exclude.
 
\subsection{The Cauchy problem at $t=0$ and propagation of singularities in the elastic case}\label{sec_GO2} 

 We consider the elastic system in $\R^3$ now. Actually, most of the analysis holds for arbitrary $n\ge3$ but since we rely on the analysis of the principal symbol of $\Lambda$ below done in three dimensions, we consider $n=3$. Since the characteristics are of constant multiplicities, this case is well understood, see, e.g., \cite{Taylor-book0} that we review below,  or \cite{Dencker_polar}. Below we give some details specific for the elastic case which allow us to describe explicitly the boundary data $f$ generating p-waves or s-waves only, up to lower order terms. 

Consider the elastic wave equation 
\be{el}
\begin{split}
u_{tt}-Eu&=0,\\
(u,u_t)|_{t=0}&=(h_1,h_2)
\end{split}
\ee
with Cauchy data $\mathbf{h}:=(h_1,h_2)$ at $t=0$. We want to solve it microlocally for $t$ in some interval and $x$ in an open set. 
Let $\Pi_p=\Pi_p(D)$ be the projection to the p-modes, i.e., $\Pi_p$ is the Fourier multiplier $\hat u\mapsto (\xi/|\xi|)[(\xi/|\xi|)\cdot \hat u] $ and let $\Pi_s=\Id-\Pi_p$. It is easy to see that 
$\Pi_s$ is the Fourier multiplier $\hat u\mapsto -(\xi/|\xi|)\times (\xi/|\xi|)\ \times \hat u$. Also, we may regard $h = \Pi_ph+\Pi_s h$ as the potential/solenoidal (or the Hodge) decomposition of the 1-form $h$, see, e.g., \cite{Sh-book}.

By  \r{s0}, 
\be{10}
 {E} = c_p^2\Delta\Pi_p+ c_s^2\Delta \Pi_s,\quad \mod \Psi^1,
\ee
where $\Psi^m$ is the class of classical \PDO s of order $m$; and we will denote by $S^m$ the corresponding symbol class. 
This shows that to construct the leading singularity of the solution, we need to solve the decoupled system
\be{10a}
\begin{split}
(\partial_t^2-c_p^2\Delta)u_p=0, & \quad(u_p, \partial_t u_p)|_{t=0} = \Pi_p \mathbf{h},\\
(\partial_t^2-c_s^2\Delta)u_s=0, &\quad  (u_s, \partial_t u_s)|_{t=0} = \Pi_s \mathbf{h}.
\end{split}
\ee
Those are two vector valued acoustic equations. The singularities propagate along unit speed  geodesics lifted to the tangent bundle (identified with the covector one for each speed) of the metrics $c_p^{-2}\d x^2$ and $c_s^{-2}\d x^2$, respectively.  
To relate \r{10a} to \r{el}, set $u_p= \Pi_pu$, $u_s=\Pi_su$. Then $Eu_p= c_p^2\Delta u_p + R_1u$, where $R_1$ is a \PDO\ of order $1$. Next, applying $\Pi_p$ to the initial conditions in \r{el}, we get the initial conditions in the first equation in \r{10a}. 
We get a similar conclusion for $u_s$. Therefore, $u_p$ and $u_s$ solve a system, compare with  \r{10a}, of the type
\[
(\partial_t^2-c_p^2\Delta)u_p +R_{11}u_p+R_{12}u_s=0, \quad (\partial_t^2-c_s^2\Delta)u_s  +R_{21}u_p+R_{22}u_s=0,
\]
where $R_{ij}$ are \PDO s of order one. Since propagation of singularities is governed by the principal part of that system only, we prove the claim associated with \r{10a}: the leading singularities, say in $H^m$ modulo $H^{m-1}$ with a fixed $m$, of $u$ can be  computed as in \r{10a}; and the whole singularities still propagate along the zero bicharacteristics of the speeds $c_p$ and $c_s$. This is a general conclusion for the solution $u$ of the elastic system since locally, we can always take the traces of $u$ and $u_t$ to some hyperplane $t=t_0$ and view the solution as the one with Cauchy data given by those traces.

We recall also the construction in \cite{Taylor-book0}, which provides another proof of the propagation of singularities in this case. 
The principal symbol $\sigma_p(-E)$ of $-E$ has eigenvalues of constant multiplicities. It is well known, see, e.g., \cite{Taylor-book0} that near every $(x_0,\xi^0)\in T^*\bar\Omega\setminus 0$, one can decouple the full symbol  $\sigma(-E)$ fully up to symbols of order $-\infty$. In other words, there exist elliptic matrix valued \PDO s $U$ and $V$ of order $0$  microlocally defined near $(x_0,\xi^0)$, so that 
\be{VEU}
VEU= \begin{pmatrix}  P&0\\0&S\end{pmatrix}
\ee
modulo $S^{-\infty}$ near $(x_0,\xi^0)$,  where $P$ is scalar, and  $S$ is a $2\times2$ matrix symbol, with principal symbols $\sigma_p(P) = c^2_p|\xi|^2$, $\sigma_p(S)=c_s^2|\xi|^2$. In other words, $P$ is scalar and $S$ is principally scalar. In fact, $U$ can be chosen to be unitary with $V=U^*$ in microlocal sense \cite{MR1777025}. As an example, the principal symbol of $U$  can be chosen to be 
\[
\sigma_p(U) =|\xi|^{-1}\begin{pmatrix}  \xi_1&  -\xi_2&0\\ \xi_2&\xi_1&\xi_3\\ \xi_3&0&-\xi_2\end{pmatrix}
\]
when $\xi_2\not=0$. 
It then follows that microlocally, the elasticity system decouples into the wave equations $(\partial_t^2-c^2\Delta)u=0$ with $c=c_p$ or $c=c_s$; the first one scalar, and the second one a $2\times 2$ system. The first one has $\Sigma_p$ as a characteristic manifold, while the second one has $\Sigma_s$. 
Even though $U$ and $V$ depend on the microlocal neighborhoods of the characteristic varieties $\Sigma_{p,s}$ we work in, the wave front sets of $U^{-1}f$, in those neighborhoods, we can apply the propagation of singularities results, or directly the microlocal geometric optics construction used below. Then we conclude that singularities in those neighborhoods propagate along the zero bicharacteristics of $\tau^2-c_p^2|\xi|^2$ and $\tau^2-c_s^2|\xi|^2$, respectively. This implies a global result, as well. The advantage of this construction is that we can do it to infinite order.

\section{Proof of the main results} 
%

The following theorem is a local version of the statement that given $\Lambda$, we can recover  the lens relations $(L_p,\ell_p)$ and $(L_s,\ell_s)$, see \cite{Rachele00, HansenUhlmann03} for the global version.

\begin{theorem}\label{thm2}
Let $\Gamma\subset\bo$ be relatively open and let $T>0$. For $0<\eps\le T$, assume that  for every   $f$ with a  singular support in $(0,\eps)\times \Gamma$,  $\Lambda f$  is known on $(0,T)\times\Gamma$, up to a smooth function. Then 
 the lens relations $(L_p,\ell_p)$ and $(L_s,\ell_s)$ are determined uniquely on the open sets of $(x,v)$ with $x\in\Gamma$ so that the unit speed geodesic issued from $(x,v)$ (i.e., unit speed in the direction of $v$) in the metric $c_p^{-2}\d x^2$, respectively $c_s^{-2}\d x^2$, is transversal at $x$ and hits $\bo$ again, transversely,  at a point in $\Gamma$ at a time not exceeding $T$. 
\end{theorem}
 
\begin{proof}
By  \cite{Rachele_2000}, the jets of $\rho,\lambda, \mu$ at $\Gamma$ are uniquely determined by the kernel of  $\Lambda$ known on $[(0,\eps)\times\Gamma]^2$ for any fixed $\eps>0$. Since the proof is based on applying $\Lambda$ to highly oscillatory functions, any smooth addition to that kernel would not change the reconstruction.

Using this, we extend  $\rho,\lambda, \mu$   smoothly to some small neighborhood $U$ of $\Gamma$ in the exterior $\Omega_{\rm ext}$ of $\bar\Omega$ in a way determined uniquely by the data in the theorem.  

Choose $\zeta_1:=(t_1,x_1,\tau_1=-c_p|\xi^1|, \xi^1)\in T^*((0,\eps)\times\bar\Omega)\setminus 0$ (which is characteristic for the Hamiltonian related to $c_p$), with $x_1\in\bo$ and $\xi^1$ pointing into $\Omega$.  Assume that the null bicharacteristic $\gamma$ in $\bar\Omega$ (actually, $T^*(\R\times\bar\Omega)$) through $\zeta_1$ is transversal at that point, and hits $T^*(\R\times\bo)$ again transversely at some $\zeta_2:= (t_2,x_2,\tau_2, \xi^2)$ with $x_2\in\bo$. Extend $\gamma$ outside the domain on the side of $\zeta_1$ until it hits $\{t=0\}$ at some point $\zeta_0:= (t=0,x_0,\tau_0=-c_p|\xi|,\xi^0)$. If $0<\eps\ll1$, this short segment will be outside $\bar\Omega$, i.e., $x_0\not\in\bar\Omega$, see Figure~\ref{pic3}.  

 Let $u_0'$ an outgoing (smooth for $t<0$ in $\Omega$) microlocal solution in $(0,T)\times\Omega_{\rm ext}$  related to $(\rho,\lambda, \mu)$, with a wave front set in $\Sigma_p$ only,  the latter supported in a 
small conic neighborhood of $\gamma$. This solution can be constructed by choosing suitable Cauchy data at $t=0$ near $x_0$ as explained in the previous section. 
We think of $u_0'$ as a microlocal p-wave propagating along $\gamma$.   We cut smoothly $u_0'$ so that its support is concentrated 
near $\gamma$; call the result $u_0$. 
Then $(\partial_t^2-E)u_0=v\in C^\infty((0,T)\times\Omega_{\rm ext})$ and $u_0=\partial_tu_0= 0$, $v=0$ for $t$ near $0$. 

The trace $f$ of $u_0$ on the boundary can be naturally written as $f=f_1+f_2$, where $f_j$ have wave front sets in small conic neighborhoods of   the projection of $\gamma$ 
on $T^*(\R\times\bo)$, i.e., close to $\zeta_j' := (t_j,x_j,\tau_j, (\xi^j)')$, $j=1,2$, where the prime stands for a tangential projection.   

Let $u$ solve \r{1} with $f=f_1$ (i.e., $f_2$ is replaced by zero). The singularities issued from $f_1$ will propagate to the future only and before they hit $\bo$ again, $u_0$ and $u$ differ by a smooth function. When they hit $\bo$, they will reflect at $\bo$ and there will be a possible mode conversion. We are not going to build a parametrix for the reflection. Instead, it is enough to prove that $\partial_\nu u$ has a non-empty wave from set in a conic neighborhood of $\zeta_2' $. 

Let $v$ be $u$ extended as zero to $\Omega_{\rm ext}$. Then $(\partial_t^2-E)v = -\rho^{-1}(N  u)\otimes  \delta_{\rm b}$, where $\delta_{\rm b}$ is the delta function on $\R\times\bo$  on the boundary, see \r{2a}. 
Assume $\zeta_2'\not\in \WF(\partial_\nu u)$; then at $(t_2,x_2)$, the wave front set of $-(N u)\otimes \delta_{\rm b}$ on the plane $\pi$ spanned by $\zeta_2'$ and the conormal to the boundary can be only along the conormal, as it follows by the calculus of the wave front sets. In particular, $(\tau_2,\xi^2)$ (the phase component of $\gamma$ at that point) cannot be in the wave front set of $ (Nu)\otimes  \delta_{\rm b}$ because the latter is in $\pi$; and $(\tau_2,\xi^2)$  is certainly not conormal, being characteristic. 
By the propagation of singularities theorem (with a source term having a wave front set away from the microlocal region of interest), each point of $\gamma$ must be a singularity for $v$, or none is. This is a contradiction since $v$ is singular on $\gamma$ in the domain, and non-singular on $\gamma$ outside it. Therefore, $\zeta_2' \in \WF(N u)$. 

We can take a sequence of $f$'s as above with shrinking wave fronts sets to $\zeta_1$ (or take $f$ with $\WF(f_1)$ on the radial ray through $\zeta_1$) to conclude that  $\Lambda$ determines the p-lens relation $(L_p,\ell_p)$   at $\zeta_0$. 
Since the part of $\gamma$ between $\zeta_0$ and $\zeta_1$ is uniquely determined by the data, conclude that $(L_p,\ell_p)$ is uniquely determined at $\zeta_1$ as well. 

To show that $\Lambda$ determines the lens relation $(L_s,\ell_s)$ related to the s-waves on $\mathcal{G}_s$, we argue as above. 
\end{proof}

\begin{figure}
\includegraphics[page=4,scale=0.7]{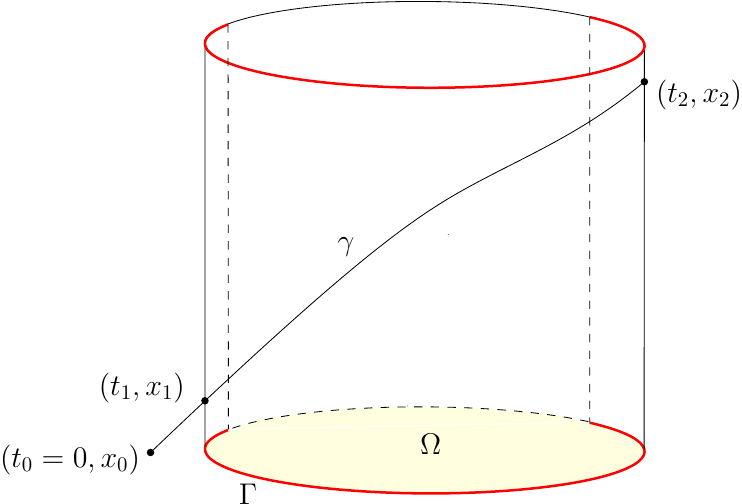}
\caption{The bicharacteristic $\gamma$ (projected to the base)}\label{pic3}
\end{figure} 

\begin{remark}\label{remark5.1}
The proof actually shows that we only  need to know the wave front set of the kernel of $\Lambda$ microlocally at $((t_1, x_1, \tau_1,-(\xi^1)'), (t_2, x_2, \tau_2,(\xi^2)')$ only for $c_p$ and similarly for $c_s$, in order to decide if   $(x_1,(\xi^1)', x_2,(\xi^2)', t_2-t_1)$   belongs to  the graph of the p-lens relation with an identification of covectors and vectors by the metric $c_p^{-2}\d x^2$.  Here, $(t_2, x_2, \tau_2,(\xi^2)')$ is the image of $(t_1, x_1, \tau_1,\xi^1)$ (recall that $(\xi^1)'$ is the projection of $\xi^1$) under the bicharacteristic flow until it hits the boundary; and then projected there. 
\end{remark}

\begin{proof}[Proof of Theorem~\ref{thm1}]
Consider $c_p$ first. By Theorem~\ref{thm2}, we can recover the lens relation $(L_p,\ell_p)$ on the foliation. By \cite{SUV_localrigidity}, this recover $c_p$ in the region covered by the foliation, as claimed. The proof for $c_s$ is the same. 
\end{proof}  

\section{The Herglotz and Wieckert \& Zoeppritz condition}\label{sec_ex}
We formulate generalized version of the Herglotz \cite{Herglotz} and Wieckert and Zoeppritz \cite{WZ} condition on a speed $c(x)$ in the ball $B(0,R)$:
\be{HWZ}
\frac{\partial}{\partial r}\frac{r}{c(r,\omega)}>0, \quad \text{for $0<r=|x|\le R$},
\ee
where $x=r\omega$, $|\omega|=1$. The original condition in  \cite{Herglotz, WZ} is about radial speeds $c(r)$ only. 
In particular, \r{HWZ} holds when $\partial_rc<0$, i.e., when the speed decreases in depth. This inequality was shown on \cite{SUV_localrigidity} to be equivalent to the requirement    the Euclidean sphere $|x|=r$ to be strictly convex with respect to $c^{-2}\d x^2$. If $\bo$ is flat locally, the convexity condition is that $c$ increases with depth. We formulate those two conditions formally in the following. 

\begin{lemma} \label{lemma_conv1}\ 

(a) The Euclidean spheres $S_r=\{x\in \R^n; \; |x|=r\}$, $R_1\le r\le R_2$, form a strictly convex foliation in some set with respect to the metric $c^{-2}\d x^2$,  viewed from the exterior, 
 if and only if \r{HWZ} holds for such $r$ and for $x=r\omega$ in that set. 

(b) The Euclidean hyperplanes $\{x\in \R^n;\; x^n=C\}$, $C_1\le C\le C_2$ form a strictly convex foliation in some set with respect to the metric $c^{-2}\d x^2$, viewed from $x^n>C_2$ if and only   $\partial c/\partial x^n>0$ in that set. 
\end{lemma}

Part (a) is proved in \cite{SUV_localrigidity}. Those two statements
are a partial case of the following more general one. Recall that
strict convexity of an oriented hypersurface $S$ in a Riemannian
manifold is defined as a positivity of the second fundamental form;
and if that form in non-negative, we will call $S$ convex. If the
second fundamental form vanishes at some point of $S$, we call this
point flat, which is a special case of convex. 
Under this definition, totally geodesic hypersurfaces are still convex.  

\begin{lemma}\label{lemma_conv2}
Let the oriented hypersurface $S$ be strictly convex w.r.t.\ the metric $g$ at some point $x_0$. Fix a smooth $c>0$.  Let $\partial/\partial\nu$ be the  unit normal derivative at $x_0$ pointing to the convex side. If $\partial c/\partial{\nu}<0$, then   $S$ is strictly convex w.r.t.\ the metric $c^{-2}g$ at $x_0$. 

If $S$ is flat at $x_0$, then $\partial c/\partial{\nu}<0$ is an if and only if condition for strict convexity. 
\end{lemma}

\begin{proof} 
We work in semigeodesic local coordinates near $x_0$ so that $S$ is given locally by $x^n=0$ and the convex side is $x^n>0$. 
We need to show that the second fundamental form related to $c^{-2}g$ on  hyperplane $x^n=0$ is positive when that related to $g$ is. Denote the Christoffel symbols of $g$ by $\Gamma_{ij}^k$ and those of $c^{-2}\d x^2$ by $\tilde \Gamma_{ij}^k$. Using the relationship between Christoffel symbols of conformal metrics, we get  
\[
\tilde \Gamma_{ij}^k = \Gamma_{ij}^k  +\frac12 c^2\left(\delta_{j}^k\partial_{x^i} + \delta_{i}^k\partial_{x^j} -  g_{ij}\partial_{x^k}\right)c^{-2}. 
\]
On  $T\{x^n=0\}$, which implies $\xi^n=0$ in particular, the second fundamental form of $g$ and  $c^{-2}g$ on $x^n=0$ are related by
\[
-\tilde\Gamma_{\alpha\beta}^n\xi^\alpha\xi^\beta = 
-\Gamma_{\alpha\beta}^n\xi^\alpha\xi^\beta  +\frac12c^2|\xi'|_g^2\partial_{x^n} c^{-2},
\]
where Greek indices run from $1$ to $n-1$ and $\xi'=(\xi^1,\dots,\xi^{n-1})$. Therefore, that form is positive   if $\partial_{x^n} c<0$; and when the form on the right vanishes, then this is an if and only of condition. 
\end{proof}

Lemma~\ref{lemma_conv1}(b) then follows from Lemma~\ref{lemma_conv2} since 
the Euclidean hypersurfaces are convex by our definition (they are flat). Using the fact that the Euclidean spheres are strictly convex, we can derive strict convexity in Lemma~\ref{lemma_conv1}(a) under the weaker condition $\partial_r c<0$. One can check directly the the Euclidean  spheres are flat  for the metric $|x|^{-2}\d x^2$, which implies Lemma~\ref{lemma_conv1}(a) in its full strength. 

\begin{figure}[h]
\includegraphics[scale=1,page=2]{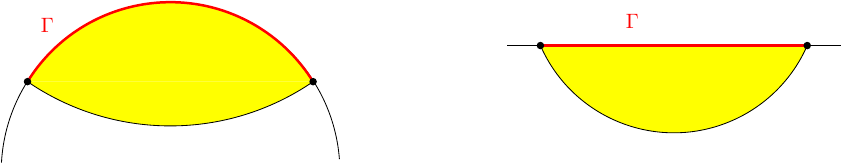}
\caption{The shaded region is where we can recover the speed if the speed increases with depth, illustrating Lemma~\ref{lemma_conv1} (a) and (b), respectively.}\label{pic2}
\end{figure}

Part (b) and Theorem~\ref{thm1} in particular provide uniqueness for the local seismology problem when the surface of the Earth is modeled as the plane $x^3=0$ in $\R^3$, and the Earth itself is given  locally by $x^3>0$, under the conditions $\partial_{x^3}c_p>0$ and $\partial_{x^3}c_s>0$, see Figure~\ref{pic2}(b). For deeper regions, the spherical model can be used and then condition \r{HWZ} guarantees existence of a strictly convex foliation. 
Those conditions are satisfied in the Upper Mantle, at least, according to the popular Preliminary Reference Earth Model (PERM) \cite{DZIEWONSKI1981297}. In fact, the stronger condition $\partial_rc<0$ holds, and   Figure~\ref{pic2}(a)  illustrates typical regions where the uniqueness holds.


\end{document}